\documentclass[]{article}
\usepackage[english]{babel}
\usepackage[normalem]{ulem}
\usepackage[all]{xy}
\usepackage{amsmath,amsthm,booktabs,color,dsfont,epsfig,graphicx,mathrsfs,multirow,scalefnt}
\usepackage{amsfonts, amssymb}

\newtheorem{theo}{Theorem}

\newtheorem{defn}[theo]{Definition}

\newtheorem{ex}[theo]{Example}

\newcommand{\Z}{\mathbb{Z}}
\newcommand{\N}{\mathbb{N}}
\newcommand{\NZ}{\mathbb{M}}

\newcommand{\s}{\sigma}

\newcommand{\AZ}{\mathcal{A}^\mathbb{Z}}

\title{A note on the definition of sliding block codes and the Curtis-Hedlund-Lyndon Theorem}

\author{
\small{Marcelo Sobottka}\\
\footnotesize{UFSC -- Department of Mathematics}\\
\footnotesize{88040-900 Florian\'{o}polis - SC, Brazil}\\
\footnotesize{\texttt{sobottka@mtm.ufsc.br}}
\and
\small{Daniel Gon\c{c}alves}\\
\footnotesize{UFSC -- Department of Mathematics}\\
\footnotesize{88040-900 Florian\'{o}polis - SC, Brazil}\\
\footnotesize{\texttt{daemig@gmail.com}}
}

\date{}

\begin{document}

\maketitle

\begin{abstract} In this note we propose an alternative definition for sliding block codes between shift spaces. This definition coincides with the usual definition in the case that the shift space is defined on a finite alphabet, but it encompass a larger class of maps when the alphabet is infinite. In any case, the proposed definition keeps the idea that a sliding block code is a map with a local rule. Using this new definition we prove that the Curtis-Hedlund-Lyndon Theorem always holds for shift spaces over countable alphabets.
\end{abstract}


\section{Introduction}

Let $\NZ$ be a monoid, and denote as $1$ the identity in $\NZ$. Given $g,h\in\NZ$, let $gh$ denote the operation of $g$ with $h$ in $\NZ$. Given a countable alphabet $A$ we consider it with the discrete topology and define $A^\NZ$ as the set of all sequences over $A$ indexed by $\NZ$. Given $\mathbf{x}=(x_i)_{i\in \NZ}\in A^\NZ$ and  $N\subset\NZ$ we denote by $\mathbf{x}_N$ the restriction of $\mathbf{x}$ to the indices in $N$, that is, $\mathbf{x}_N:=(x_i)_{i\in N}\in A^N$.

On $A^\NZ$ we consider the product topology, which makes $A^\NZ$ a Hausdorff and totally disconnected topological space. Given a finite subset $D\subset \NZ$, we define the cylinder given by $D$ and $(a_i)_{i\in D}\in A^{|D|}$ as the set
$$\big[(a_i)_{i\in D}\big]_D:=\{\mathbf{x}\in A^\NZ:\ x_i=a_i,\ \forall i\in D\}.$$
When $D=\{g\}$ we will denote the cylinder which fix the symbol $b$ in the entry $g$ simply as $[b]_g$. Notice that each cyclinder above is clopen and, furthermore, the collection of all cylinders as above forms a basis for the produt topology in $A^\NZ$.

Given $g\in\NZ$, the translation on $A^\NZ$ defined by $g$ is the map $\s^g:A^\NZ\to A^\NZ$, given by $\s^g\big((x_i)_{i\in\NZ}\big)=(x_{gi})_{i\in\NZ}$. A shift space over the alphabet $A$ is a set $\Lambda\subset A^\NZ$ which is closed under the topology of $A^\NZ$ and  invariant by translations, that is, $\s^g(\Lambda)\subset \Lambda$ for all $g\in\NZ$. We consider on a shift space $\Lambda$ the induced topology from $A^\NZ$. Given a finite set of indexes $N\subset\NZ$, let $W_N(\Lambda)\subset A^N$ denote the set of all finite configurations of $A^N$ that appear in some sequence of $\Lambda$, that is,

$$W_N(\Lambda):=\{(w_i)_{i\in N}\in A^N:\ \exists \ x\in\Lambda \text{ s.t. } x_i=w_i\ \forall i\in N\}.$$

Note that, if $\Lambda$ is a shift space and $M\subset\NZ$ is a translation of $N\subset\NZ$, that is $M=gN$ for some $g\in\NZ$, then $(w_j)_{j\in M}\in W_M(\Lambda)$ if, and only if, there exists $(v_i)_{i\in N}\in W_N(\Lambda)$ such that  $w_j=v_i$ whenever $j=gi$. In other words, $W_M(\Lambda)$ and $W_N(\Lambda)$ contain the same configurations modulo translation.\\

\begin{defn}[Classical sliding block codes]\label{SBC_classic} Let $A$ and $B$ be two countable alphabets and let  $\Lambda\subset A^\NZ$ be a shift space. A map $\Phi:\Lambda\to B^\NZ$ is a sliding block code if there exists a finite set of indexes $N\subset\NZ$ and a local rule $\phi:W_N(\Lambda)\to B$ such that, for all $g\in\NZ$ and $\mathbf{x}=(x_i)_{i\in\NZ}\in\Lambda$, it holds that $$\big(\Phi(\mathbf{x})\big)_g=\phi\big(\s^g(\mathbf{x})_N\big).$$
\end{defn}

Intuitively speaking, we say that $\Phi$ is a sliding block code if the symbol $\big(\Phi(\mathbf{x})\big)_g$ is a function on the configuration of $\mathbf{x}$ at a finite number of indexes (more specifically at the indexes in $gN$, called the neighborhood of $x_g$).
When $A$ is finite, the Curtis-Hedlund-Lyndon Theorem (see \cite[Theorem 1.8.1]{Ceccherini-Silberstein--Coornaert}) states that a map $\Phi:\Lambda\to B^\NZ$ is a sliding block code if, and only if, it is continuous and invariant by translations (that is, $\Phi\circ\s^g=\s^g\circ \Phi$ for all $g\in\NZ$). However, if $A$ is infinite, then it is possible to construct continuous shift-commuting maps which do not satisfy Definition \ref{SBC_classic}. For instance, the map given in Example \ref{example1} below is continuous and invariant by translations, but it is not possible to describe it via a local rule as in Definition \ref{SBC_classic}. This difference between the finite alphabet case and the infinite alphabet case arises from the fact that in the first case shift spaces are always compact spaces, while in the second case shift spaces are always non-compact (frequently they are not even locally compact).

In \cite{Ceccherini-Silberstein--Coornaert}, the authors give a version of the Curtis-Hedlund-Lyndon Theorem for maps defined on shift spaces over infinite alphabets. More specifically, it was proved that a map defined on a shift space over an infinite alphabet satisfies Definition \ref{SBC_classic} if, and only if, it is uniformly continuous and shift commuting (see \cite[Theorem 1.9.1]{Ceccherini-Silberstein--Coornaert}).

In this paper, building from the ideas in \cite{GSS}, where it was proved a version of the Curtis-Hedlund-Lyndon Theorem for the Ott-Tomforde-Willis compactification of one-sided shift spaces over infinite alphabets \cite{Ott_et_Al2014} and from the ideas in \cite{GSS1}, where some weaker versions of Curtis-Hedlund-Lyndon Theorem were proved for a compactification for two-sided shift spaces over infinite alphabets, we propose an alternative definition for sliding block codes (see Definition \ref{SBC_alternative}). This new definition coincides with the classical definition when defined on shift spaces over finite alphabets, but enlarges the class of sliding block codes when dealing with shift spaces over infinite alphabets. More specifically, we will consider maps $\Phi:\Lambda\to B^\NZ$ such that, for all $\mathbf{x}\in\Lambda$ and  $g\in\NZ$, the symbol $\bigl(\Phi(\mathbf{x})\bigr)_g$ depends only on a finite number of entries
$(x_j)_{j\in I}$, for some finite set of indexes $I\subset\NZ$ that depend only on the configuration of $\mathbf{x}$ ``around'' $x_g$. Such maps also have local rules, however these local rules are defined using a family of possible neighborhoods. The inspiration for considering these types of maps comes from the notion of variable length Markov processes, see \cite{Rissanen1983}. Our proposed generalization of sliding block codes becomes natural once we rewrite the definition of classical sliding block codes in an equivalent way, so that we can see classical sliding block codes exactly as the maps $\Phi:\Lambda\to B^\NZ$ such that, for all $g\in\NZ$, the map $\big(\Phi(\cdot)\big)_g:\Lambda\to B$ is a simple function (see Section \ref{Generalized_SBC}).

\section{Generalized sliding block codes}\label{Generalized_SBC}

Suppose that $A$ is finite and $\Phi:\Lambda\to B^\NZ$ is a sliding block code. It follows from the Curtis-Hedlund-Lyndon Theorem that $\Phi$ is continuous and shift commuting. So, for each symbol $b\in B$, the set $C_b:=\Phi^{-1}([b]_{\mathbf{1}_\NZ})$ is a clopen set of $\Lambda$ and therefore it can be written as a (possibly empty) union of disjoint cylinders of $\Lambda$. In particular, since $\Lambda$ is compact, each $C_b$ is a finite (possibly empty) union of cylinders of $\Lambda$ and, furthermore, $\{C_b\}_{b\in B}$ is a partition of $\Lambda$. Now, since $\Phi$ is shift commuting, for each $g\in\NZ$ we have that $\s^g(\mathbf{x})\in C_b$ if, and only if, $\big(\Phi(\mathbf{x})\big)_g=\big(\s^g(\Phi(\mathbf{x}))\big)_{\mathbf{1}_\NZ}=\big(\Phi(\s^g(\mathbf{x}))\big)_{\mathbf{1}_\NZ}=b$. In other words, the local rule $\phi$ of $\Phi$ is defined exactly by the words of $\Lambda$ that define the sets $C_b$ and it is implicitly given by the simple function $$\big(\Phi(\cdot)\big)_g=\sum_{b\in B} b\mathbf{1}_{C_b}\circ \s^g(\cdot),$$ where each $C_b$ is a finite union of cylinders of $\Lambda$, the sum stands for the symbolic sum and $\mathbf{1}_{C_b}$ denotes the characteristic function of the se $C_b$.

The above observation leads us to propose that sliding block codes should not be defined as maps whose local rules have a bound on the number of entries used to define them, since this is just a consequence of the compactness of the shift space. The fundamental feature of a sliding block code is the fact that, for each $g\in\NZ$, the map $\big(\Phi(\cdot)\big)_g:\Lambda\to B$ is a simple function which does not depend on $g$ and such that, to decide the image of any $\mathbf{x}\in\Lambda$, one just needs to know a finite (yet variable) number of entries which form the neighborhood of $x_g$.

\begin{defn}[Generalized sliding block codes]\label{SBC_alternative}
Let $A$ and $B$ be two countable alphabets and let  $\Lambda\subset A^\NZ$ be a shift space. A map $\Phi:\Lambda \to B^\NZ$ is a {\em generalized sliding block code} if there exists $\{C_b\}_{b\in B}$ a partition of $\Lambda$ where each nonempty $C_b$ is a union of cylinders of $\Lambda$, such that
\begin{equation}\label{LR_block_code}\bigl(\Phi(\mathbf{x})\bigr)_g=\sum_{b\in B}b\mathbf{1}_{C_b}\circ\sigma^g(\mathbf{x}),\quad \forall \ \mathbf{x}\in\Lambda,\ \forall \ g\in\NZ, \end{equation} where  $\mathbf{1}_{C_b}$ is the
characteristic function of the set $C_b$ and $\sum$ stands for the symbolic sum.
\end{defn}

For a better illustration of Definition \ref{SBC_alternative}, suppose that $\NZ$ is equipped with a metric $\mathbf{d}$ which is invariant by translations (this is the case, for example, when $\NZ=\N^d$ or $\NZ=\Z^d$ for some integer $d\geq 1$). Then for a classical sliding block code $\Phi:\Lambda\to B^\NZ$, given by a local rule $\phi$ on $W_N(\Lambda)$, we define the radius of the sliding block code as $$r(\Phi):=\max_{i\in N}\mathbf{d}(i,1).$$
Notice that, due to the invariance by translation of the metric, for all $\mathbf{x}\in\Lambda$ and  $g\in\NZ$, to compute the value of $\big(\Phi(\mathbf{x})\big)_g$ it is sufficient to known all entries $x_j$ such that $\mathbf{d}(j,g)\leq r(\Phi)$. On the other hand, if $\Phi$ is a generalized sliding block code then it may not exist a value $r(\Phi)$ with the property above. However, Definition \ref{SBC_alternative} implies that, for all $\mathbf{x}\in\Lambda$ and $g\in\NZ$, there exists $$r(\mathbf{x},g):=\max_{i\in D}\mathbf{d}(i,1),$$ where $D$ is the set of indeces such that $\s^g(\mathbf{x})\in [(a_i)_{i\in D}]_D$ and $[(a_i)_{i\in D}]_D$ is a cylinder in some $C_b$. Therefore, to compute the value of  $\big(\Phi(\mathbf{x})\big)_g$, it is sufficient to known all entries $x_j$ such that $\mathbf{d}(j,g)\leq r(\mathbf{x},g)$.
In other words, while a classical sliding block code has a fixed radius that can be used for all $\mathbf{x}\in\Lambda$ and $g\in\NZ$, a generalized sliding block code has a variable radius, whose length depends on the configuration near $x_g$.
Next we give an example of a generalized sliding block code with variable radius and an example of a map which is not a generalized sliding block code and for which there exists $\mathbf{x}$ such that, to compute $\big(\Phi(\mathbf{x})\big)_g$ for any $g\in\NZ$, one needs to know infinite entries of $\mathbf{x}$.

\begin{ex}\label{example1} Denote by $\N$ the set of all nonnegative integers with the usual sum, suppose $\NZ:=A:=\N$, and consider the map $\Phi:A^\N\to A^\N$ defined by \begin{equation}\label{eq:example1}\big(\Phi(\mathbf{x})\big)_j=x_{j+x_j},\qquad \forall j\in\NZ.\end{equation} It follows that $\Phi$ is a generalized sliding block code where for each $b\in A$,
 $$\begin{array}{rcl}
 C_0&:=&\displaystyle [0]_0,\\\\
 \text{and, for all } b\neq 0,\\\\
  C_b&:=&\displaystyle\bigcup_{n\geq 1}\big[(w_i)_{i\in\{0,n\}}\big]_{\{0,n\}},\qquad \text{where}\quad w_0:=n\text{ and } w_n:=b\ \text{ for all } n\geq 1
 \end{array}
 $$

  Furthermore, it is direct that $\Phi$ is continuous and shift commuting.

\end{ex}

\begin{ex}
Let $A:=\NZ:=\N$ and let $\{A_\ell\}_{\ell\in\N}$ be a partition of $A$ into finite sets such that at least one of them has two or more elements. Let $\Lambda:=\bigcup_{\ell\in\N}A_\ell^\N$ and $\Phi:\Lambda\to\AZ$ be the map given by $$\big(\Phi(\mathbf{x})\big)_j=\max_{i\geq j}x_i.$$ It follows that $\Phi$ is not a generalized sliding block code and furthermore is not continuous (although it is shift commuting).

\end{ex}

\section{Curtis-Hedlund-Lyndon Theorem for generalized sliding block codes\sloppy}

In \cite[Theorem 1.9.1]{Ceccherini-Silberstein--Coornaert} it was proved that, whatever the cardinality of A is,  a map $\Phi:\Lambda\to B^\NZ$ is a sliding block code (according to Definition \ref{SBC_classic}) if, and only if, it is uniformly continuous and shift commuting. If $A$ is a finite alphabet then $\Lambda$ is compact and it follows that the family of continuous maps coincides with the family of uniformly continuous maps. Also due to the compactness of the space, the continuity of a map implies that there exists $N\subset \NZ$ such that, for all $\mathbf{x}\in\Lambda$ and $g\in\NZ$, the symbol $\big(\Phi(\mathbf{x})\big)_g$ depends only on the configuration $\mathbf{x}_{gN}$. In other words, the existence of a local rule based on a single neighborhood is linked to the uniform continuity of the map, which is an automatic consequence of the finiteness of the alphabet. However, for shift spaces over infinite alphabets, there exist continuous maps which are not uniformly continuous and, therefore are not classical sliding block codes. Using the notion of generalized sliding block codes, we can state a more general version of the Curtis-Hedlund-Lyndon Theorem:\\

\begin{theo} A map $\Phi:\Lambda\subset A^\NZ \to B^\NZ$ is a generalized sliding block if, and only if, it is continuous and commutes with all translations.
\end{theo}

\begin{proof}
Suppose $\Phi:\Lambda\to B^\NZ$ is continuous and commutes with all translations. For each $b\in B$, define $C_b:=\Phi^{-1}([b]_{1})$. It follows that $\{C_b\}_{b\in B}$ is a cover of $\Lambda$. Furthermore, since $[b]_{1}$ is clopen in $B^\NZ$ and $\Phi$ is continuous, then each $C_b$ is clopen in $\Lambda$ and therefore, if $C_b$ is not empty, it can be written as a union of cylinders of $\Lambda$. Now, if $\mathbf{x}\in\Lambda$ belongs to some set $C_b$ then $\big(\Phi(\mathbf{x})\big)_{1}=b$, that is, we have that
$\bigl(\Phi(\cdot)\bigr)_{1}=\sum_{b\in B}b\mathbf{1}_{C_b}(\cdot)$. Thus, since $\Phi$ commutes with all translations, it follows that
$\big(\Phi(\cdot)\big)_g=\big(\s^g\circ\Phi(\cdot)\big)_{1}=\big(\Phi\circ \s^g(\cdot)\big)_{1}=\sum_{b\in B}b\mathbf{1}_{C_b}\circ\s^g(\cdot)$ for all $g\in\NZ$.\\

Now suppose that $\Phi$ is a generalized sliding block code. First let us show that $\Phi$ is invariant by translations, that is, for all $h\in\NZ$ we have that $\s^h\circ\Phi=\Phi\circ\s^h$. In fact, given $h\in\NZ$, it follows that for any $\mathbf{x}\in\Lambda$ and any $g\in\NZ$ we have
$$\begin{array}{lcl}\big(\s^h\circ\Phi(\mathbf{x})\big)_g&=&\big(\s^g\circ\s^h\circ\Phi(\mathbf{x})\big)_{1}=\big(\s^{gh}\circ\Phi(\mathbf{x})\big)_{1}
=\big(\Phi(\mathbf{x})\big)_{gh}\\\\&=&\sum_{b\in B}b\mathbf{1}_{C_b}\circ\s^{gh}(\mathbf{x})
=\sum_{b\in B}b\mathbf{1}_{C_b}\circ\s^g\circ\s^h(\mathbf{x})=\big(\Phi\circ\s^h(\mathbf{x})\big)_{g}.\end{array}$$
To prove that $\Phi$ is continuous, notice that, if $E\subset\NZ$ is a finite set and $[(b_i)_{i\in E}]_E$ is a cylinder of $B^\NZ$ then
$$\Phi^{-1}\big([(b_i)_{i\in E}]_E\big)=\Phi^{-1}\left(\bigcap_{i\in E}[b_i]_i\right)=\bigcap_{i\in E}\Phi^{-1}\left([b_i]_i\right)$$
and hence it is enough to show that, for $i \in \NZ$, $\Phi^{-1}([b]_i)$ is an open set.
Now, since $\Phi$ is invariant by translations, for all $i\in\NZ$ and $b\in B$ we have that
$[b]_i=(\s^i)^{-1}\big([b]_{1}\big)$. Therefore
$\Phi^{-1}([b]_i)=\Phi^{-1}\circ(\s^{i})^{-1}\big([b]_{1}\big)=(\s^i)^{-1}\circ\Phi^{-1}\big([b]_{1}\big)=(\s^i)^{-1}(C_b)$,
and hence, since $C_b$ is a union of cylinders of $\Lambda$ and $\s^i$ is continuous for all $i\in\NZ$ (see \cite[Proposition 1.2.2]{Ceccherini-Silberstein--Coornaert}), we have that $(\s^i)^{-1}(C_b)$ is also an union of cylinders of $\Lambda$ and so it is open as desired.

\end{proof}


\section*{Acknowledgments}

\noindent M. Sobottka was supported by CNPq-Brazil grants 304813/2012-5 and 480314/2013-6. The first version of this paper was carried out while the author was invited as researcher visitor at PIPGEs/UFSCar and ICMC/USP-S\~{a}o Carlos. The author thanks both institutions and their respective graduate programs for the financial support and hospitality.

\noindent D. Gon\c{c}alves was partially supported by Capes grant PVE085/2012 and CNPq.


\end{document}